\documentclass[11pt,a4paper]{article}

\usepackage[english]{babel}
\usepackage[T1]{fontenc}
\usepackage[utf8]{inputenc}
\usepackage[
  margin=2.5cm,
  includefoot,
  footskip=30pt,
]{geometry}
\usepackage{lmodern}
\usepackage{amsmath}
\usepackage{amssymb}
\usepackage{amsthm}
\usepackage{booktabs}
\usepackage{array}
\usepackage{graphicx}
\usepackage{cite}
\usepackage[dvipsnames]{xcolor}
\usepackage{tikz-cd}
\usetikzlibrary{decorations.markings,angles}
\usepackage{stmaryrd} \SetSymbolFont{stmry}{bold}{U}{stmry}{m}{n}
\usepackage{mathrsfs}
\usepackage{extarrows}
\usepackage{mathtools}
\usepackage{tensor}
\usepackage{enumerate}
\usepackage{adjustbox}

\usepackage[mathcal]{eucal}
\usepackage{hyperref}
\hypersetup{
	colorlinks,
	urlcolor={red!55!black},
	citecolor={green!60!black},
	linkcolor={red!55!black}
}
\usepackage{bm}
\usepackage{tabularx}
\usepackage{cleveref}
\usepackage[strict]{changepage}
\usepackage{tensor}

\mathchardef\mhyphen="2D
\def\on{\operatorname}

\makeatletter
\providecommand{\leftsquigarrow}{%
  \mathrel{\mathpalette\reflect@squig\relax}%
}
\newcommand{\reflect@squig}[2]{%
  \reflectbox{$\m@th#1\rightsquigarrow$}%
}
\makeatother

\definecolor{ao}{rgb}{0.0, 0.5, 0.0}

\newtheorem{theorem}{Theorem}[section]
\newtheorem{lemma}[theorem]{Lemma}

\newtheorem{proposition}[theorem]{Proposition}
\newtheorem{corollary}[theorem]{Corollary}

\theoremstyle{definition}
\newtheorem{construction}[theorem]{Construction}
\newtheorem{definition}[theorem]{Definition}

\newtheorem{remark}[theorem]{Remark}
\newtheorem{example}[theorem]{Example}

\newcommand\noloc{%
  \nobreak
  \mspace{6mu plus 1mu}
  {:}
  \nonscript\mkern-\thinmuskip
  \mathpunct{}
  \mspace{2mu}
}

\newcommand{\A}{\mathcal{A}}
\newcommand{\B}{\mathcal{B}}
\newcommand{\C}{\mathcal{C}}
\newcommand{\D}{\mathcal{D}}

\newcommand{\op}{\on{op}}

\title{$\infty$-categorical group quotients via skew group algebras}
\author{Merlin Christ}
\date{\today}

\begin{document}

\maketitle

\abstract{We relate group quotients of dg-categories and linear stable $\infty$-categories. Given a group acting on a dg-algebra, we prove that the skew group dg-algebra is Morita equivalent to the dg-categorical homotopy group quotient. We also treat the cases of group actions on dg-categories, with corresponding skew group dg-categories, and of orbit dg-categories. Finally, we describe a version of the skew group algebra in the setting of ring spectra and relate it with $\infty$-categorical group quotients.
}

\tableofcontents

\section{Introduction}

\subsubsection*{Skew group algebras and skew group dg-categories}

The skew group algebra is a classical construction in the representation theory of finite dimensional algebras \cite{RR85} and also appears in the construction of non-commutative resolutions of singularities \cite{KV00,VdB22}. As input for its construction serves an algebra $A$ over a field $k$ and a group $G$ acting on $A$ by automorphisms. The skew group algebra $AG$ has as underlying vector space the coproduct $A^{\amalg G}$, with multiplication determined on additive generators by $(g,a)\cdot (h,b)=(gh, a\cdot (g.b))$. An analogue of this construction is also known in the case that $A$ is a dg-algebra, dg-category or even an $A_\infty$-category, see for instance \cite{LM20,OZ22,AP24}.

The representation theory of the skew group algebra $AG$ is typically better behaved than that of the algebra of fixed points $A^G$, which is why $AG$ has become a standard construction. Indeed, under some finiteness assumptions, the module category $\on{mod}(AG)$ is equivalent to the category $\on{mod}(A)^G$ of $G$-equivariant $A$-modules, see \cite{Dem11}. 

By passing to the setting of dg-categories, we can clarify the universal property of the skew group algebra:

\begin{proposition}[\Cref{prop:skewgroupcatiscolim} and \Cref{rem:htpycolim}]\label{prop:dgcolimBG}
Fix a base field $k$. Let $A$ be a dg-category with a strict action by a group $G$. We consider the $G$-action as a functor 
\[ \rho\colon BG\to \on{dgCat}_k\,,\]
where $BG$ is the classifying space of $G$. The homotopy colimit\footnote{With respect to the Morita model structure on the category $\on{dgCat}_k$ of dg-categories.} of $\rho$ is Morita equivalent to the skew group dg-category, denoted $A\ast G$. 

If $A$ has a single object, we can identify it with a dg-algebra and in this case $A\ast G=AG$ describes a dg-version of the skew group algebra.
\end{proposition}

The above homotopy colimit is also called the homotopy group quotient of $A$ by $G$. By using homotopy colimits, instead of homotopy limits, we can avoid any finiteness assumptions on $G$. Note also that we impose no assumptions on the characteristic of the ground field $k$. 

In the case $G=\mathbb{Z}$, we will also treat non-strict actions and show that the orbit dg-category, recently constructed for non-strict $\mathbb{Z}$-actions in \cite{FKQ24}, describes the homotopy colimit in \Cref{prop:orbitdgcat}. 

To prove \Cref{prop:dgcolimBG}, we construct a strict cocone under a diagram equivalent to $\rho\colon BG\to \on{dgCat}_k$, whose tip is Morita equivalent to $AG$. Passing to derived $\infty$-categories, we obtain an induced functor from the $\infty$-categorical colimit $\D(A)_G$ to the derived $\infty$-category of $AG$. We use the $\infty$-categorical colimit $\D(A)_G$ to describe the homotopy colimit, via the equivalence of dg-categories up to Morita equivalence with compactly generated $k$-linear stable $\infty$-categories. This functor $\D(A)_G\to \D(AG)$ is an equivalence, which is shown by comparing the derived endomorphisms of generators.

We also give a separate proof of \Cref{prop:dgcolimBG} under some cofibrancy assumptions, by exhibiting the skew-group algebra as the strict dg-categorical colimit of a cofibrant diagram, see \Cref{prop:homotopycolim}.\\

The results of this paper unify different perspectives on categorical group quotients, leading to new results about these. For instance, exhibiting $\infty$-categorical universal properties of the dg-categorical constructions makes these accessible to powerful $\infty$-categorical arguments. An immediate consequence of the above results is that the passage to skew group dg-algebras and skew group dg-categories commutes with homotopy colimits of dg categories, as colimits commute with colimits. Reversing the perspective, the results of this note also yield concrete models for the abstract $\infty$-categorical constructions. We illustrate this by describing periodic derived $\infty$-categories and periodic topological Fukaya categories of surfaces in terms of orbit categories in \Cref{subsec:periodicex}. We also obtain new results on group quotients and module $1$-categories, see below.

\subsubsection*{Group actions on linear stable $\infty$-categories}

Let $G$ be a group. The classifying space $BG$ is the category with a single object $\ast$ with endomorphisms $G$. We work over a base $\mathbb{E}_\infty$-ring spectrum $R$ and denote $\on{Mod}_R$ denote the symmetric monoidal stable $\infty$-category of $R$-module spectra. We denote by 
\[ \on{LinCat}_{\on{Mod}_R}=\on{LMod}_{\on{Mod}_R}(\mathcal{P}r^L)\]
the $\infty$-category of $\on{Mod}_R$-linear presentable $\infty$-categories (these are automatically stable). 

The standard way to define an action of $G$ on a $\on{Mod}_R$-linear $\infty$-category $\C$ is as a functor $BG\to \on{LinCat}_{\on{Mod}_R}$, mapping $\ast$ to $\C$. Central to our treatment of group actions will be the following equivalent way to encode such $G$-actions, see for instance \cite{CCRY22,BCSY23}: forming the coproduct of $\on{Mod}_R$ over $G$ yields a monoidal $\infty$-category $\on{Mod}_R^{\amalg G}\in \on{LinCat}_{\on{Mod}_R}$, called the categorical group algebra; a $G$-action on $\C$ can equivalently be expressed as a left $\on{Mod}_R^{\amalg G}$-action on $\C$.

This perspective allows to apply the powerful and well developed framework of $\infty$-categorical algebra objects and modules of \cite{HA} to the study of group actions and group quotients. For instance, the $\infty$-categorical colimit over $BG$ amounts in terms of left actions to tensoring with the $\on{Mod}_R$-$\on{Mod}_R^{\amalg G}$-bimodule $\on{Mod}_R$. Using this perspective, desired basic properties of group quotients readily follow.

\subsubsection*{Skew group ring spectra}

In \Cref{sec:skewgroupalg}, we generalize the construction of the skew group algebra to the setting of a group $G$ acting on an $R$-linear ring spectrum $A$. The skew group algebra $AG$ is then an $R$-linear ring spectrum with underlying spectrum the coproduct $A^{\amalg G}$, and the multiplication generalizing the above. The actual construction of the ring spectrum $AG$ is based on universal constructions, see \Cref{rem:skewgroupspectrum} for a summary. \Cref{prop:dgcolimBG} then generalizes as follows:

\begin{theorem}\label{thm:colimBG}
Let $R$ be the base $\mathbb{E}_\infty$-ring spectrum and $A\in \on{Alg}(\on{Mod}_R)$ be an $R$-linear ring spectrum. Let $\rho\colon BG\to \on{Alg}(\on{Mod}_R)$ be an action of a group $G$ on $A$. The colimit of the functor between $\infty$-categories
\begin{equation}\label{eq:fun1} BG\xlongrightarrow{\rho} \on{Alg}(\on{Mod}_R)\xlongrightarrow{\on{RMod}_{(\mhyphen)}} \on{LinCat}_{\on{Mod}_R}\end{equation}
is equivalent to $\on{RMod}_{AG}$.
\end{theorem}

We remark that the limit of the functor \eqref{eq:fun1} is equivalent to its colimit, since the theorem is formulated in the setting of presentable $\infty$-categories, and hence to $\on{RMod}_{AG}$, see \Cref{lem:limit}.

\subsection*{Comparison with previous results}\label{introsec:previousresults}

\subsubsection*{Classical skew group algebras}

Given a category $C$ with an action by a group $G$, we denote by $C^G$ the category of $G$-equivariant objects. Note that $C^G$ describes the homotopy limit of a functor $BG\to \on{Cat}$, mapping the basepoint of $BG$ to $C$, see \Cref{rem:equivariant=lim}. Given a (not necessarily finite dimensional) algebra $A$ over a commutative ring $k$, we denote by $\on{Mod}(A)$ the (non-small) abelian category of right $A$-modules. Using \Cref{thm:colimBG} and \Cref{lem:limit}, we obtain the following:

\begin{corollary}\label{cor:modAG}
Let $k$ be a commutative ring and $G$ a group acting on a $k$-algebra $A$. There exists an equivalence of derived $\infty$-categories 
\[ \D(A\ast G)\simeq \D(A)^G\coloneqq \on{lim}_{BG}\D(A)\,.\]
Passing to homotopy $1$-categories, this equivalence restricts to an equivalence of abelian subcategories $\on{Mod}(A\ast G)\simeq \on{Mod}(A)^G$. 

If furthermore $G$ is finite, $k$ a field, and $A$ finite dimensional over $k$, then this equivalence restricts to the abelian categories of finite dimensional modules: $\on{mod}(A\ast G)\simeq \on{mod}(A)^G$. 
\end{corollary}

The equivalence of the small module categories from \Cref{cor:modAG} was shown in the case that the characteristic of $k$ does not divide the order of $G$ in \cite[Prop.~2.48]{Dem11}, but appears to be new in this generality. Corresponding equivalences of the bounded derived categories were deduced in \cite[Thm.~7.1]{Ela14} or \cite{Che15,Bal11}. 

The proof of \Cref{cor:modAG} is given in \Cref{subsec:quotientdgcat}.

\subsubsection*{Ring spectra}

Previous results in the literature concern a version of \Cref{thm:colimBG} in the case that $A=R$ is the base ring spectrum and $G$ is a monoid in the $\infty$-category of spaces (generalizing groups). The restriction $A=R$ corresponds to the assertion that the functor $BG\to \on{LinCat}_{\on{Mod}_R}$ is pointed. We refer to \cite[Prop.~3.13]{Dou05} or \cite[Thm.~0.0.7]{HM23} in case of $R=\mathbb{S}$ the sphere spectrum, and \cite[Thm.~7.13]{CCRY22} for general $R$. The analogue of the skew group algebra in this setting is also called the Thom spectrum.

\subsection*{Acknowledgements}

I thank Patrick Le Meur for help in navigating the literature on skew group algebras. I thank Bernhard Keller for suggesting \Cref{lem:cofinal} and further comments. I also thank Bastiaan Cnossen for explaining their proof of \cite[Lem.~4.49]{CCRY22} and helpful remarks. Finally, I thank Claire Amiot, Norihiro Hanihara, and Pierre-Guy Plamondon for inspiring discussions.

This project has received funding from the European Union’s Horizon 2020 research and innovation programme under the Marie Skłodowska-Curie grant agreement No 101034255. The author is a member of the Hausdorff Center for Mathematics at the University of Bonn (DFG GZ 2047/1, project ID 390685813).

\section*{Higher categorical preliminaries}

We generally follow the notation and conventions of \cite{HTT,HA}.

\begin{itemize}
\item We consider dg-categories over a fixed base field $k$. The derived $\infty$-category $\D(\B)$ of a dg-category $\B$ is defined as the $\on{Ind}$-completion of the dg-nerve $N_{\on{dg}}(\on{Perf}(\B))$ of the dg-category of cofibrant compact right dg-$B$-modules. The passage to the derived $\infty$-category defines a colimit preserving and reflecting functor between $\infty$-categories
\[
\D(\mhyphen)\colon N(\on{dgCat}_k)[M^{-1}]\to \on{LinCat}_{\on{Mod}_k}\,,
\]
with $M$ the collection of Morita equivalences, see also \Cref{rem:htpycolim}.
\item For a discussion of $\infty$-operads, algebra objects and modules, we refer to Chapter 2 and 4 of \cite{HA}.
\item $R$ will usually denote a base $\mathbb{E}_\infty$-ring spectrum. Given a $\on{Mod}_R$-linear $\infty$-category $\D$ and $X,Y\in \D$, we write $\on{Mor}_\D(X,Y)\in \on{Mod}_R$ for the $\on{Mod}_R$-linear morphism object, see \cite[Def.~4.2.1.28]{HA}.
\item Given a regular cardinal $\kappa$, we denote by $\on{LinCat}_{\on{Mod}_R}^{\kappa}= \on{Mod}_{\on{Mod}_R}(\mathcal{P}r^L_{\kappa})$ the presentable $\infty$-category of $\on{Mod}_R$-modules in the presentable $\infty$-category $\mathcal{P}r^L_{\kappa}$ of $\kappa$-compactly generated presentable $\infty$-categories, as in \cite[Notation~5.3.2.8]{HA}. Its presentability is the advantage of $\on{LinCat}_{\on{Mod}_R}^{\kappa}$ over $\on{LinCat}_{\on{Mod}_R}$. 

Given a $\on{Mod}_R$-linear monoidal $\infty$-category $\C\in \on{Alg}(\on{LinCat}_{\on{Mod}_R})$, we choose a sufficiently large regular cardinal $\kappa$ as in \cite[Lem.~5.3.2.12]{HA}, for which in particular $\C$ is $\kappa$-compactly generated. We also denote 
\[ 
\on{LinCat}_{\C}^\kappa=\on{LMod}_{\C}(\on{LinCat}_{\on{Mod}_R}^{\kappa})
\]
for the presentable $\infty$-category of $\kappa$-compactly generated $\C$-linear categories. 
\item Given a presentable monoidal $\infty$-category $\C$, in Sections 4.8.3-4.8.5 of \cite{HA}, Lurie describes a fully faithful (see \cite[Thm.~4.8.5.5]{HA}) functor\footnote{More precisely, the functor $\Theta_\ast$ considered here is a restriction the functor $\Theta_\ast$ of \cite{HA}, having additional functoriality in $\C$. We leave the choice of $\C$ in the notation for $\Theta_\ast$ implicit.} 
\[ \Theta_\ast\colon \on{Alg}(\C)\longrightarrow (\on{LinCat}_{\C})_{\C/}\,,\] 
mapping an algebra object $A\in \C$ to its $\infty$-category $\on{RMod}_A(\C)$ of right modules, together with the functor $\mhyphen\otimes A\colon \C\to \on{RMod}_A(\C)$ mapping the monoidal unit to $A$. 

The image of the functor $\Theta_*$ is concretely characterized in \cite[Prop.~4.8.5.8]{HA}. The functor $\Theta$ is defined as the composite of $\Theta_\ast$ with the forgetful functor 
\[ (\on{LinCat}_{\on{Mod}_R})_{\on{Mod}_R/}\to \on{LinCat}_{\on{Mod}_R}\,.\]
\end{itemize}

\section{\texorpdfstring{$\infty$}{Infinity}-categorical group actions}\label{sec:catgroquot}

For the entire section, we fix a group $G$ (which is not required to be finite). 

\subsection{The categorical group algebra}

Given a set (or a space) $X$, we denote by $\on{Mod}_R^{\amalg X}$ the colimit over $X$ of the constant diagram in $\on{LinCat}_{\on{Mod}_R}$ with value $\on{Mod}_R$. Note that $\on{Mod}_R^{\amalg X}\simeq \on{Mod}_R^{\times X}$ also describes the limit over $X$, which follows from the equivalence $\mathcal{P}r^L\simeq (\mathcal{P}r^R)^{\on{op}}$. 

There is an essentially unique colimit preserving functor
\[ \on{Mod}_R^{\amalg (\mhyphen)}\colon \mathcal{S}\to \on{LinCat}_{\on{Mod}_R}=\on{Mod}_{\on{Mod}_R}(\mathcal{P}r^L)\,,\]
determined by mapping $\ast$ to $\on{Mod}_R$. Furthermore, since $\mathcal{S}$ is the unit object in the symmetric monoidal $\infty$-category $\mathcal{P}r^L$, the functor $\on{Mod}_R^{\amalg}$ lifts essentially uniquely to a symmetric monoidal functor $(\on{Mod}_R^{\amalg (\mhyphen)})^{\otimes}\colon \mathcal{S}^{\otimes}\to  \on{LinCat}_{\on{Mod}_R}^{\otimes}$, see \cite[Prop.~3.2.1.8]{HA}.

\begin{construction}\label{constr:catgroupalg}
The group $G$ gives rise to an $\on{Assoc}$-algebra object in the Cartesian symmetric monoidal $\infty$-category $\on{Set}^{\otimes}$ of sets, i.e.~a morphism of $\infty$-operads $\on{Assoc}^{\otimes}\to \on{Set}^{\otimes}$ over $N(\on{Fin}_\ast)$, see also \cite[Def.~4.1.1.3]{HA} for the definition of $\on{Assoc}^{\otimes}$.

Composing with the inclusion of symmetric monoidal $\infty$-categories $\on{Set}^{\otimes}\subset \mathcal{S}^{\otimes}$ and the symmetric monoidal functor $(\on{Mod}_R^{\amalg (\mhyphen)})^{\otimes}$, we obtain an $\on{Assoc}$-algebra object $\on{Mod}_R^{\amalg G}\in \on{Alg}_{\on{Assoc}}(\on{LinCat}_{\on{Mod}_R})$, meaning a $\on{Mod}_R$-linear monoidal $\infty$-category.

We note that any group homomorphism $G\to G'$ induces a monoidal functor $\on{Mod}_R^{\amalg G}\to \on{Mod}_R^{\amalg G'}$.
\end{construction}

\begin{definition}\label{def:catgroupalg}
The monoidal $\infty$-category $\on{Mod}_R^{\amalg G}$ from \Cref{constr:catgroupalg} is called the categorical group algebra of $G$.
\end{definition}

Given $g\in G$, we denote the object of $\on{Mod}_R^{\amalg G}$ lying in the $g$-th component with value $R$ by $R_g$. The monoidal unit of $\on{Mod}_R^{\amalg G}$ is given by $R_e$.

\begin{remark}\label{rem:psi}
The unique group homomorphism $\psi\colon G\to \{e\}$ gives by \Cref{constr:catgroupalg} rise to a monoidal functor $\psi\colon \on{Mod}_R^{\amalg G}\to \on{Mod}_R$. Explicitly, $\psi$ maps a $G$-tuple $(A_g)_{g\in G}$ to $\coprod_{g\in G} A_g$.
\end{remark}

\begin{lemma}\label{lem:locrigid}~
\begin{enumerate}[(1)]
\item The monoidal $\infty$-category $\on{Mod}_R^{\amalg G}$ is locally rigid in the sense of \textnormal{\cite[Def.~D.7.4.1]{SAG}}.
\item If $F\colon \C\to \D$ is a $\on{Mod}_R^{\amalg G}$-linear functor whose right adjoint $G$ preserves colimits, then $G$ is also $\on{Mod}_R^{\amalg G}$-linear.
\end{enumerate}
\end{lemma}

\begin{proof}
We begin with showing the local rigidity. Conditions (1) and (2) of Definition D.7.4.1 are clear. The unit object of $\on{Mod}_R^{\amalg G}$ is given by the object $R_{e}$, with $e$ the unit of $G$, and clearly compact, giving condition (3). 

The subset of left and right dualizable objects is closed under finite limits and colimits. For condition (4), it thus suffices to check that each object $R_g$ with $g\in G$ admits a left and right dual. Since $R_g$ is invertible with inverse $R_{g^{-1}}$, this is clear.

Part (2) of the Lemma follows from part (1) by \cite[Rem.~D.7.4.4]{SAG}. The broken reference at the end of loc.~cit.~may refer to the fact that the $R$-linear structure on $G$ in question is induced by the $R$-linear structure of $F$ via \cite[Cor.~7.3.2.12]{HA}, so that the existence of the $R$-linear structure in question is a property.
\end{proof}

\begin{example}\label{ex:adjpsi}
We can consider the monoidal functor $\psi\colon \on{Mod}_R^{\amalg G} \to \on{Mod}_R$ as a $\on{Mod}_R^{\amalg G}$-linear functor. Its right adjoint $\psi^R\colon \on{Mod}_R\to \on{Mod}_R^{\amalg G}$, $G$-componentwise given by the identity functor, inherits by \Cref{lem:locrigid} the structure of a $\on{Mod}_R^{\amalg G}$-linear functor. 

The composite $\psi^R\circ \psi\colon \on{Mod}_R^{\amalg G}\to \on{Mod}_R^{\amalg G}$ is equivalent to $\coprod_{g\in G} R_g\otimes (\mhyphen)$ and inherits the structure of a monad. 

If $G$ is a finite group, then the right adjoint $\psi^{RR}$ of $\psi^R$ agrees with $\psi$. The inherited $\on{Mod}_R^{\amalg G}$-linear structure of $\psi^{RR}$ also coincides with that of $\psi$. This follows from the fact that $\on{Mod}_R^{\amalg G}$ is the free $\on{Mod}_R^{\amalg G}$-linear $\infty$-category generated by a point, so that both linear functors are determined by the value of the monoidal unit $R_e$. 
\end{example}

\subsection{Group actions and group quotients}\label{subsec:grpquot} \label{subsec:groupactionviacatgrpalg}

We denote by $BG$ the classifying space of $G$, meaning the $1$-category with a unique object $\ast$ and endomorphisms $G$. We will not distinguish between $BG$ and its nerve $N(BG)\in \on{Set}_\Delta$ in notation. 

A $G$-action on an object $\C$ in an $\infty$-category $\mathscr{C}$ is understood to be a functor $BG\to \mathscr{C}$, mapping $\ast$ to $\C$. 

\begin{definition}
Let $\mathscr{C}$ be an $\infty$-category that admits limits and colimits. Given $G$-action $\rho\colon BG\to \mathscr{C}$ on $\C\in \mathscr{C}$, we call 
\begin{itemize}
\item the limit $\C^G\coloneqq \on{lim}(\rho)$ the fixed points of the $G$-action on $\C$.
\item the colimit $\C_G\coloneqq \on{colim}(\rho)$ the group quotient of $\C$ by $G$.
\end{itemize}
\end{definition}

\begin{lemma}\label{lem:limit}
Consider a functor $\rho_\C\colon BG\to \on{LinCat}_{\on{Mod}_R}$ describing a $G$-action on a $\on{Mod}_R$-linear $\infty$-category $\C$. Its colimits $\C_G$ is equivalent to its limit $\C^G$ in $\on{LinCat}_{\on{Mod}_R}$. 
\end{lemma}

\begin{proof}
The conservative functor $\on{LMod}_{\on{Mod}_R^{\amalg G}}(\mathcal{P}r^L)\to \mathcal{P}r^L$ reflects limits and colimits and preserves both by \cite[Cor.~3.4.3.6,~3.4.4.6]{HA}. The right adjoint diagram of the colimit diagram of $\rho_\C$ again lies in $\on{LMod}_{\on{Mod}_R^{\amalg G}}(\mathcal{P}r^L)$ by \Cref{lem:locrigid}. It is a limit diagram, since its underlying diagram in $\mathcal{P}r^L$ is, as follows from the duality $\mathcal{P}r^L\simeq (\mathcal{P}r^R)^{\on{op}}$ and the fact that the functors $\mathcal{P}r^R,\mathcal{P}r^L\to \on{Cat}_\infty$ preserve and reflect limits. Note that $BG^{\op}\simeq BG$, since $BG$ is a space, hence the right adjoint diagram of $\rho_\C$ is equivalent to $\rho_\C$. 
\end{proof}

We choose a sufficiently large regular cardinal $\kappa$. Then $G$-actions on $\kappa$-compactly generated $\on{Mod}_R$-linear $\infty$-categories can be equivalently expressed as follows:

\begin{proposition}[$\!\!$\cite{CCRY22,BCSY23}]\label{prop:Gaction} 
There exists a canonical equivalence of $\infty$-categories 
\[
\on{Fun}(BG,\on{LinCat}_{\on{Mod}_R}^\kappa) \simeq \on{LinCat}_{\on{Mod}_R^{\amalg G}}^\kappa
\,.
\] 
\end{proposition}

\begin{proof}
This is \cite[Lem.~4.49]{CCRY22} applied in the case $\mathscr{C}=\on{LinCat}_{\on{Mod}_R}^\kappa$ and $A=G$.
\end{proof}

\begin{remark}
The relation between a $G$-action $\rho\colon BG\to \on{LinCat}_{\on{Mod}_R}^\kappa$ and the $\on{Mod}_R^{\amalg G}$-linear structure on a given $\on{Mod}_R$-linear $\infty$-category $\C$ can be explicitly described as follows: given $g\in G$, its action $\rho(g)\colon \C\simeq \C$ is equivalent to the functor $R_g\otimes \mhyphen\colon \C\simeq \C$. 
\end{remark}

\begin{lemma}[$\!\!${\cite[Lem.~4.51]{CCRY22}}]\label{lem:colimastensorproduct}
Consider a group action $BG\to \on{LinCat}_{\on{Mod}_R}$ on $\C$ and the corresponding $\on{Mod}_R^{\amalg G}$-linear structure of $\C$ of \Cref{prop:Gaction}.
\begin{enumerate}[(1)]
\item The group quotient of $\C$ is equivalent to the relative tensor product:
\begin{equation}\label{eq:grp_quot_=_tensor_product}\C_G\simeq \on{Mod}_R \otimes_{\on{Mod}_R^{\amalg G}} \C\in \on{LinCat}_{\on{Mod}_R}\,.\end{equation}
\item The functor $\C\to \C_G$ appearing in the colimit cone is equivalent to $\psi \otimes_{\on{Mod}_R^{\amalg G}}\C$, with $\psi$ as in \Cref{rem:psi}.
\end{enumerate}
\end{lemma}

\begin{proof}
The $\on{Mod}_R^{\amalg G}$-linear $\infty$-category $\on{Mod}_R^{\amalg G}$ inherits a $G$-action by \Cref{prop:Gaction}, satisfying that $\on{colim}_{BG}\on{Mod}_R^{\amalg G}\simeq \on{Mod}_R$. Using that the relative tensor product preserves colimits in the second entry, we find that 
\[
\C_G\simeq \on{colim}_{BG} (\on{Mod}_R^{\amalg G}\otimes_{\on{Mod}_R^{\amalg G}}\C)\simeq \on{Mod}_R\otimes_{\on{Mod}_R^{\amalg G}}\C \,,
\]
showing statement (1). Statement (2) amounts to the fact that $\psi\colon \on{Mod}_R^{\amalg G}\to \on{Mod}_R$ is the functor from the colimit cone of $\on{Mod}_R^{\amalg G}$. 

The statements is also noted in \cite[Lem.~4.51]{CCRY22}, using the fact that the inclusion $\on{LinCat}_{\on{Mod}_R}^\kappa\subset \on{LinCat}_{\on{Mod}_R}$ preserves colimits \cite[Lem.~5.3.2.9]{HA}. 
\end{proof}

\begin{lemma}\label{lem:monad}
Consider a group action $\rho_\C\colon BG\to \on{LinCat}_{\on{Mod}_R}$  on $\C$ and the corresponding $\on{Mod}_R^{\amalg G}$-linear structure.
\begin{enumerate}[(1)]
\item The functor 
\begin{equation}\label{eq:monadic} \psi^R\otimes_{\on{Mod}_R^{\amalg G}} \C\colon \C_G\simeq \on{Mod}_R \otimes_{\on{Mod}_R^{\amalg G}} \C \longrightarrow \on{Mod}_R^{\amalg G} \otimes_{\on{Mod}_R^{\amalg G}} \C\simeq \C\end{equation}
is monadic. 
\item The endofunctor of $\C$ underlying the adjunction monad $\psi^R\psi\otimes_{\on{Mod}_R^{\amalg G}} \C$ of the adjunction $\psi\otimes_{\on{Mod}_R^{\amalg G}} \C\dashv \psi^R\otimes_{\on{Mod}_R^{\amalg G}} \C$ is given by $\coprod_{g\in G}\rho_\C(g)$.
\item If $G$ is finite, then the functor \eqref{eq:monadic} is also left adjoint to $\C\otimes_{\on{Mod}_R^{\amalg G}} \psi\colon \C\to \C_G$.
\end{enumerate}
\end{lemma}

\begin{proof}
We begin with showing (1). By \Cref{lem:limit}, $\C_G$ is equivalent to the limit of $\rho_\C$, and hence to the $\infty$-category of coCartesian sections of the Grothendieck construction of $\rho_\C$, see \cite[\href{https://kerodon.net/tag/05RX}{Prop.~05RX}]{Ker}. Under this equivalence, the functor 
\[ \C\otimes_{\on{Mod}_R^{\amalg G}} \psi^R\colon \C_G\simeq \on{colim}(\rho_\C)\longrightarrow \C\]
evaluates a coCartesian section at $\ast \in BG$. Thus, the functor is clearly conservative and hence monadic by \cite[Thm.~4.7.3.5]{HA}. 

For (2), we note the equivalence
\[
\psi^R\psi\otimes_{\on{Mod}_R^{\amalg G}}\C\simeq \coprod_{g\in G}\left( R_g\otimes (\mhyphen)\right) \otimes_{\on{Mod}_R^{\amalg G}} \C \simeq \coprod_{g\in G}\rho_\C(g)\,.
\]

For (3), we note that by \Cref{ex:adjpsi}, if $G$ is finite, the $\on{Mod}_R^{\amalg G}$-linear functor $\psi^R$ is also left adjoint to $\psi$, which induces the desired adjunction.
\end{proof}

\section{Skew group dg-categories}\label{sec:skewgrpdgcat}

In this section, we consider a dg-category $A$ with a strict action by a group $G$. In \Cref{subsec:quotientdgcat}, we relate the skew group dg-category with the $\infty$-categorical group quotient. In \Cref{subsec:homotopycolim}, we describe under some conditions the skew group dg-category as the colimit of a cofibrant diagram. 

\subsection{Skew group dg-categories as $\infty$-categorical group quotients}\label{subsec:quotientdgcat}

We fix a group $G$ with a strict action on a dg-category $A$. We can equivalently consider the action as a functor $\rho_A\colon BG\to \on{dgCat}_k$ mapping $\ast$ to $A$. We define the skew group dg-category $A\ast G$ as follows, analogous to the definition of the skew group $A_\infty$-category in \cite[Def.~5.6]{OZ22}, \cite[Def.~2.11]{AP24}. 
\begin{itemize}
\item The set of objects of $A\ast G$ is given by the set of objects of $A$. We will however write $\tilde{X}\in A\ast G$ for the object corresponding to $X\in A$.
\item Given $\tilde{X},\tilde{Y}\in A\ast G$, the morphism chain complex is given by
\[ \on{Map}_{A\ast G}(\tilde{X},\tilde{Y})=\coprod_{g\in G}\on{Map}_{A}(g.{X},{Y})\,.\]
Given a morphism $a\colon g.X\to Y$ in $A$, we write $(g,a)\colon \tilde{X}\to \tilde{Y}$ for the corresponding morphism in $A\ast G$. Every morphism in $A\ast G$ can be uniquely written as a sum $\sum_{g\in G}(g,a_g)$, with finitely many non-zero summands. 
\item the composition map 
\[ (\mhyphen)\circ (\mhyphen)\colon\on{Map}_{A\ast G}(\tilde{X},\tilde{Y})\times  \on{Map}_{A\ast G}(\tilde{Y},\tilde{Z})\longrightarrow \on{Map}_{A\ast G}(\tilde{X},\tilde{Z})\]
is defined on generators by $(g_2,b)\circ (g_1,a)\coloneqq (g_2g_1, b\circ (g_2.a))$. 
\end{itemize}

There is an apparent dg-functor $F_A\colon A \to A\ast G$, given by the assignments $X\mapsto \tilde{X}$ on objects and $a\mapsto (e,a)$ on morphisms, where $e\in G$ is the unit element. 

We can define a $G$-action $\rho_{A\ast G}\colon BG\to \on{dgCat}_k$ on $A\ast G$ by letting $h\in G$ act as $h.\tilde{X}=\widetilde{h.X}$ and $h.(g,a)=(hgh^{-1},h.a)$. Indeed, we find:
\begin{align*}
h.\left((g_2,b)\circ (g_1,a)\right) & =(hg_2g_1h^{-1},h.(b\circ (g_2.a)))\\
&  = (hg_2h^{-1},h.b)\circ (hg_1h^{-1},h.a)= h.(g_2,b)\circ h.(g_1,a) 
\end{align*} 

The dg functor $A\to A\ast G$ is clearly $G$-equivariant, which we can express as follows:

\begin{lemma}\label{lem:Gactionstrafo}
The dg functor $F_A \colon A\to A\ast G$ extends to a morphism $\rho_{A}\to \rho_{A\ast G}$ in $\on{Fun}(BG,\on{dgCat}_k)$.
\end{lemma}

\begin{definition}\label{def:reducedskewgroupcat}
We choose for every $G$-orbit $[X]$ of objects in $A$ an arbitrary representative $X\in A$. The dg-category $(A\ast G)^{\on{red}}$ is defined as the full dg-subcategory of $A\ast G$ on the objects of the form $\tilde{X}$, with $X$ a chosen representative.
\end{definition}

We note that the inclusion $(A\ast G)^{\on{red}}\subset A\ast G$ is an equivalence of dg-categories.

\begin{lemma}\label{lem:trivialG-action}
Suppose that the action of $G$ is free on the set of objects of $A$. Then the $G$-action on $(A\ast G)^{\on{red}}$ induced by the equivalence of dg-categories $(A\ast G)^{\on{red}}\simeq A\ast G$ is trivial.
\end{lemma}

\begin{proof}
We unravel the induced $G$-action on $(A \ast G)^{\on{red}}$. Fix $h\in G$. The action of $h$ is given by the composite
\[
(A\ast G)^{\on{red}}\simeq A\ast G \xlongrightarrow{\rho_{A\ast G}(h)} A\ast G \simeq (A\ast G)^{\on{red}}\,.
\]
This action clearly acts as the identity on objects. Given a morphism $(g,a)\colon \tilde{X}\to \tilde{Y}$ in $(A\ast G)^{\on{red}}$, the action of $h$ is given by the composite
\[
\tilde{X}\xlongrightarrow{(h,\on{id}_X)} \widetilde{h.X}\xlongrightarrow{(hgh^{-1},h.a)} \widetilde{h.Y} \xlongrightarrow{(h^{-1},\on{id}_Y)} \tilde{Y}
\]
which is by definition again given by $(g,a)$. 
\end{proof}

We can always arrange the $G$-action on $A$ to be free on the set of objects:

\begin{lemma}\label{lem:freeaction}
The dg-category $A$ with the group action by $G$ is equivalent in $\on{Fun}(BG,\on{dgCat}_k)$ to a dg-category $A'$ with a $G$-action that is free on its set of objects.
\end{lemma}

\begin{proof}
The set $\on{ob}(A)$ of objects of $A$ splits into the orbits of the $G$-action. We choose a set of representatives $R\subset \on{ob}(A)$ of the orbits $\on{ob}(A)/G$. We define the set of objects $\on{ob}(A')$ of $A'$ to be pairs $(x,g)$ with $x\in R$ and $g\in G$. The morphism complexes are defined as $\on{Map}_{A'}((x,g),(y,h))=\on{Map}_A(g.x,h.y)$ with composition as in $A$. The apparent dg-functor $\pi\colon A'\to A$, given on objects by $(x,g)\to g.x$ is fully faithful and essentially surjective, and hence an equivalence of dg-categories. Using the $G$-action on $A$, we define a $G$-action on $A'$, given on objects by $h.(x,g)=(x,hg)$ and on morphisms complexes as for $A$. The equivalence of dg-categories $\pi$ is clearly $G$-equivariant.
\end{proof}

Let $A'$ be as in \Cref{lem:freeaction}. Then $(A'\ast G)^{\on{red}}$ is by \Cref{lem:trivialG-action} the tip of a cocone under the diagram $BG\to \on{dgCat}_k$ describing the $G$-action on $A'$. Passing to derived $\infty$-categories, we see that 
\[ \D((A'\ast G)^{\on{red}})\simeq \D(A'\ast G)\simeq \D(A\ast G)\]
defines the tip of a cocone under the diagram 
\[ \rho_{\D(A)}\colon BG\to \on{LinCat}_{\on{Mod}_k}
\]
describing the $G$-action on $\D(A)$. We denote the colimit of the functor $\rho_{\D(A)}$ by $\D(A)_G$. By its universal property, there is an induced functor 
\begin{equation}\label{eq:infty_colim_to_derived_skew_group} \zeta\colon  \D(A)_G\longrightarrow \D(A\ast G)\end{equation}

\begin{proposition}\label{prop:skewgroupcatiscolim}
The functor $\zeta$ from \eqref{eq:infty_colim_to_derived_skew_group} is an equivalence of $k$-linear $\infty$-categories.
\end{proposition}

\begin{remark}\label{rem:htpycolim}
The $\infty$-category underlying the model category $\on{dgCat}_k$ with the Morita model structure is equivalent to the $\infty$-category $\on{LinCat}_k^{\on{cpt-gen}}$ of $k$-linear, compactly generated, presentable, and stable $\infty$-categories as well as compact objects preserving, left adjoint functors, see \cite{Coh13}. In particular, homotopy colimits in $\on{dgCat}_k$ are described by $\infty$-categorical colimits in $\on{LinCat}_k^{\on{cpt-gen}}$, see \cite[Rem.~7.9.10]{Cis}. Note further that the forgetful functor $\on{LinCat}_k^{\on{cpt-gen}}\to \on{LinCat}_k$ preserves and reflects colimits. Thus, by \Cref{prop:skewgroupcatiscolim}, $\on{Perf}(A\ast G)$ describes the homotopy colimit of $\rho_A$. 
\end{remark}

\begin{proof}[Proof of \Cref{prop:skewgroupcatiscolim}]
By construction, there is a commutative diagram of compact objects preserving $k$-linear functors,
\begin{equation}\label{eq:diagramzeta}
\begin{tikzcd}
                             & \D(A) \arrow[rd, "\D(F_A)"] \arrow[ld, "F"'] &              \\
\D(A)_G \arrow[rr, "\zeta"] &                                                & \D(A\ast G)
\end{tikzcd}
\end{equation}
where $F$ denotes the functor contained in the colimit diagram. Note that $\D(A)$ is generated under colimits by the objects arising from the objects in $A$, which are automatically compact. Similarly, the images of the objects in $A$ under $F$ and $\D(F_A)$ compactly generate, using that $F$ is monadic by \Cref{lem:monad} and \cite[Prop.~4.7.3.14]{HA}, respectively, that $F_A$ is essentially surjective. It hence suffices to show that for all $X,Y\in A$, the morphism
\begin{equation}\label{eq:morzeta}
\on{Mor}_{\D(A)_G}(F(X),F(Y))\to \on{Mor}_{\D(A\ast G)}(\D(F_A)(X),\D(F_\C)(Y))
\end{equation}
is an equivalence. We have equivalences 
\begin{equation}\label{eq:MorCG1}
\on{Mor}_{\D(A)_G}(F(X),F(Y))\simeq \on{Mor}_{\D(A)}(X,F^RF(Y))\simeq \on{Mor}_{\D(A)}(X,\coprod_{g\in G}g.Y)\simeq \coprod_{g\in G}\on{Map}_A(g.X,Y)\,,
\end{equation}
where the last equivalence uses that $X\in \D(A)$ is compact, 
and  
\begin{equation}\label{eq:MorCG2}
\on{Mor}_{\D(A\ast G)}(\D(F_A)(X),\D(F_A)(Y))\simeq \on{Map}_{A\ast G}(\tilde{X},\tilde{Y})=\coprod_{g\in G}\on{Map}_A(g.X,Y)\,.
\end{equation}
By the commutativity of the diagram \eqref{eq:diagramzeta}, the morphism \eqref{eq:morzeta} restricts to an equivalence on the $\on{Map}_{A}(X,Y)$-summands of \eqref{eq:MorCG1} and \eqref{eq:MorCG2}. The morphism \eqref{eq:morzeta} also restricts to equivalences on the other summands which follows from combining the following two observations: Firstly, for every $g\in G$, the equivalence $F(g.X)\simeq F(X)$ is mapped to an equivalence $\widetilde{g.X}\simeq \D(F_A)(g.X)\simeq \tilde{X}=\D(F_A)(X)$. Secondly, together with the morphisms coming from $A$, these equivalences generate all morphisms between $F(X),F(Y)$ in $\D(A)_G$ and between $\D(F_A)(X),\D(F_A)(Y)$ in $\D(A\ast G)$. 
\end{proof}

\begin{remark}
In the case that $A$ is an $A_\infty$-category equipped with a strict group action as in \cite{OZ22,AP24}, the above construction easily adapts to show that the skew group $A_\infty$-category $A\ast G$ of \cite{OZ22,AP24} is equivalent to the $\infty$-categorical colimit over $BG$ in the $\infty$-category $\on{Cat}_{A_\infty}[M^{-1}]$ of $A_\infty$-categories localized at Morita equivalences. For this, we use the equivalences of $\infty$-categories $\on{Cat}_{A_\infty}[M^{-1}]\simeq \on{dgCat}_k[M^{-1}]\simeq \on{LinCat}_k^{\on{cpt-gen}}$, see \cite{Pas24,COS24} and \cite{Coh13}.
\end{remark}

\begin{proof}[Proof of \Cref{cor:modAG}.]
The derived $\infty$-categories $\D(A)$ and $\D(AG)$ of the algebras are equivalent to the module $\infty$-categories $\on{RMod}_A$ and $\on{RMod}_{AG}$, respectively. Using \Cref{thm:colimBG} and \Cref{lem:limit}, we thus have an equivalence of $\infty$-categories:
\[
\D(AG)\simeq \on{lim}_{BG}\D(A)=\D(A)^G\,.
\] 
By \Cref{rem:equivariant=lim}, $N(\on{Mod}(A)^G)$ embeds fully faithfully into $\D(A)^{G}$. The nerve $N(\on{Mod}(AG))$ also embeds fully faithfully into $\D(AG)$ as the standard heart. It remains to note that these two full subcategories are identified under the above equivalence. 

An object in $\D(AG)$ lies in the heart if and only if its image under the monadic functor $\on{RHom}(AG,\mhyphen)\colon \D(AG)\to \D(A)$ lies in the heart of $\D(A)$. Under the equivalence with the limit $\D(A)^G$, this functor corresponds to the functor in the limit cone. An object in $\D(A)^G$ corresponds to a coCartesian section of the Grothendieck construction over $BG$ \cite[\href{https://kerodon.net/tag/05RX}{Prop.~05RX}]{Ker}, and lies in the image of $N(\on{Mod}(A)^G)$ if and only if its evaluation at the unique object $\ast \in BG$ lies in $N(\on{Mod}(A))\subset \D(A)$. The evaluation functor at $\ast \in BG$ on coCartesian sections also describes the functor $\D(A)^G\to \D(A)$ in the limit cone. Hence, both full subcategories have the same essential images.  
\end{proof}

\begin{remark}\label{rem:equivariant=lim}
We describe the relation of the $1$-category of equivariant objects with limits in the $\infty$-category $\on{Cat}_\infty$ of $\infty$-categories. Let $\on{Cat}$ denote the $1$-category of $1$-categories. Consider the adjunction 
\[ h(\mhyphen)\colon  \on{Set}_\Delta\leftrightarrow \on{Cat}\noloc N(\mhyphen)\] 
between the homotopy category functor and the simplicial nerve functor. Note that $N(\mhyphen)$ is fully faithful. This adjunction is a Quillen adjunction with respect to the Joyal model structure on $\on{Set}_\Delta$ and the standard model structure on $\on{Cat}$. We hence obtain an adjunction between $\infty$-categories $h(\mhyphen)\colon \on{Cat}_\infty \leftrightarrow \on{Cat}[W^{-1}]\noloc N(\mhyphen)$, where $W$ denotes the collection of equivalences of $1$-categories. The derived functor $N(\mhyphen)$ is also fully faithful. Since it preserves $\infty$-categorical limits, and we see that limits in $\on{Cat}_\infty$ of diagrams valued in nerves of $1$-categories are equivalent to nerves of $1$-categories, and hence determined by their homotopy categories. Given a diagram $\rho\colon BG\to \on{Cat}$, $\ast \to C$, one finds that the homotopy category of the limit of $N\circ \rho\colon BG\to \on{Cat}_\infty$ is equivalent to the category $C^G$ of $G$-equivariant objects of $C$, which can be seen using the explicit description of limits via coCartesian sections of the Grothendieck construction \cite[\href{https://kerodon.net/tag/05RX}{Prop.~05RX}]{Ker}. Thus, we have 
\[ N(C^G)\simeq \on{lim}_{BG} N\circ \rho \in \on{Cat}[W^{-1}]\subset \on{Cat}_\infty
\,.\]
\end{remark}

\subsection{Skew group dg-categories as homotopy colimits of cofibrant diagrams}\label{subsec:homotopycolim}

In this section, we describe an alternative way to relate skew group dg-categories with group quotients using purely model categorical techniques.

We consider the functor category $\on{Fun}(BG,\on{dgCat}_k)$ as equipped with the projective model structure inherited from the quasi-equivalence model structure on $\on{dgCat}_k$. The weak equivalences are thus the pointwise quasi-equivalences and the fibrations are the pointwise fibrations. 

\begin{lemma}\label{lem:cofibrant}
Consider a functor $\rho_A\colon BG\to\on{dgCat}_k$ describing the action of a group $G$ on a dg-category $A$. If $A$ is cofibrant and the action free on the set of objects of $A$, then $\rho_A\in \on{Fun}(BG,\on{dgCat}_k)$ is a cofibrant object with respect to the projective model structure. 
\end{lemma}

\begin{proof}
Consider an acyclic fibration $\pi\colon B\to C$ and a morphism $F\colon A\to C$ in $\on{Fun}(BG,\on{dgCat}_k)$. To prove the cofibrancy of $\rho_A$, we must find a lift $A\to B$ of $F$ along $\pi$ in $\on{Fun}(BG,\on{dgCat}_k)$.  

Since $A$ is cofibrant in $\on{dgCat}_k$, we can choose a lift $\tilde{F}$ of $F$ along $\pi$ in $\on{dgCat}_k$. We may choose the lift $\tilde{F}$ such that it is $G$-equivariant on the level of objects: we fix an object $X\in A$ in each $G$-orbit of $A$. If for $g\in G$, we have $\tilde{F}(g.X)\not =g.\tilde{F}(X)$, we choose an equivalence $\tilde{F}(g.X)\simeq g.\tilde{F}(X)$ lifting the identity $F(g.X)=g.F(X)$ and redefine $\tilde{F}(g.X)$ as $g.\tilde{F}(X)$ and change the action on morphisms using this equivalence. 

We choose a set $R$ of representatives of the $G$-orbits of the objects of $A$. We define a new dg-functor $F'\colon A\to B$ as follows:
\begin{itemize}
\item On objects, $F'$ is defined as $\tilde{F}$.
\item For $X,Y\in R$ and $g_1,g_2\in G$, we define $F'\colon \on{Map}_{A}(g_1.X,g_2.Y)\to \on{Map}_{B}(g_1.F'(X),g_2.F'(Y))$ on $\alpha\in \on{Map}_A(g_1.X,g_2.Y)$ as $g_1.\tilde{F}(g_1^{-1}.\alpha)$.
\end{itemize}

We first note that $F'$ defines a lift of $F$ along $\pi$ in $\on{dgCat}_k$: on objects that is clear and on a morphism $\alpha\in \on{Map}_A(g_1.X,g_2.Y)$, we have 
\[ \pi\circ F'(\alpha)=\pi(g_1.\tilde{F}(g^{-1}_1.\alpha))=g_1.\pi(\tilde{F}(g_1^{-1}.\alpha))=g_1.F(g_1^{-1}.\alpha)=F(\alpha)\]
using the $G$-equivariance of $F$ and $\pi$. 

Finally, we check that $F'$ is indeed $G$-equivariant, i.e.~defines the desired lift in $\on{Fun}(BG,\on{dgCat}_k)$. On objects, this was noted above. Consider a morphism $\alpha\in \on{Map}_A(g_1.X,g_2.Y)$ and $h\in G$. Then 
\[
F'(h.\alpha)=(hg_1).\tilde{F}((hg_1)^{-1}.(h.\alpha))=h.(g_1.\tilde{F}(g_1^{-1}.\alpha))=h.F'(\alpha)\,,
\] 
as desired.
\end{proof}

\begin{proposition}\label{prop:homotopycolim}
Under the assumptions of \Cref{lem:cofibrant}, the dg-category $(A\ast G)^{\on{red}}$ from \Cref{def:reducedskewgroupcat} is the tip of a (strictly commuting) homotopy colimit cocone under the functor $\rho_A$ with respect to the quasi-equivalence model structure on $\on{dgCat}_k$.
\end{proposition}

\begin{proof}
It is straightforward to see that the cocone is a colimit cocone in the $1$-category $\on{dgCat}_k$. Since the diagram $\rho_A$ is cofibrant by \Cref{lem:cofibrant}, this colimit cocone is also a homotopy colimit cocone.
\end{proof}

\section{Orbit categories}\label{sec:orbit}

In this section, we will specialize to the case where $G=\mathbb{Z}$. The arising group quotients are known as orbit categories. After a general discussion in the $\infty$-categorical context in \Cref{subsec:inftyorbitcats}, we consider non-strict $\mathbb{Z}$-action on dg-categories in \Cref{subsec:orbitdg}. Finally, we describe examples arising from periodic derived categories in \Cref{subsec:periodicex}.

\subsection{Orbit $\infty$-categories}\label{subsec:inftyorbitcats}

Let $L$ denote the simplicial set with a unique $0$-simplex $\ast$ and a unique non-degenerate $1$-simplex $1\colon \ast \to \ast$.

\begin{definition}
Let $F\colon \C\to \C$ be a $\on{Mod}_R$-linear endofunctor of a $\on{Mod}_R$-linear $\infty$-category $\C$. The $\on{Mod}_R$-linear orbit $\infty$-category $\C/F$ is defined as the colimit of the functor
\begin{equation}\label{eq:Faction} L\longrightarrow \on{LinCat}_{\on{Mod}_R},\quad \ast \mapsto \C,\quad (1\colon \ast\to \ast) \mapsto F\,.\end{equation}
\end{definition}

\begin{remark}
The orbit $\infty$-category $\C/F$ is also equivalent to the limit of the functor \eqref{eq:Faction} and hence to the $\infty$-category of coCartesian sections of the Grothendieck construction of the functor \eqref{eq:Faction}. An object of $\C/F$ thus amounts to an object $X\in \C$ together with an equivalence $X\simeq F(X)$ in $\C$. 

The lax limit of the functor \eqref{eq:Faction}, denoted $\C/^{\on{lax}}F$, is given by the $\infty$-category of all sections of the Grothendieck construction of \eqref{eq:Faction}. Its objects are given by objects $X\in \C$ together with a morphism $X\to F(X)$ in $\C$. 
\end{remark}

\begin{lemma}\label{lem:cofinal}
Let $F\colon \C\to \C$ be a $\on{Mod}_R$-linear endofunctor.
\begin{enumerate}
\item[(1)] The inclusion of simplicial sets $L\subset B\mathbb{N}, 1\mapsto 1$ is inner anodyne. Pulling back along this inclusion hence gives rise to a commutative diagram of $\infty$-categories:
\[
\begin{tikzcd}
{\on{Fun}(B\mathbb{N},\on{LinCat}_{\on{Mod}_R})} \arrow[rd, "\on{colim}"'] \arrow[rr, "\simeq"] &                          & {\on{Fun}(L,\on{LinCat}_{\on{Mod}_R})} \arrow[ld, "\on{colim}"] \\
                                                                                                & \on{LinCat}_{\on{Mod}_R} &                                                                
\end{tikzcd}
\]
\item[(2)] Pulling back along the inclusion $B\mathbb{N}\subset B\mathbb{Z}$ yields a commutative diagram of $\infty$-categories,
\[
\begin{tikzcd}
{\on{Fun}(B\mathbb{Z},\on{LinCat}_{\on{Mod}_R})} \arrow[rd, "\on{colim}"'] \arrow[rr, hook] &                          & {\on{Fun}(B\mathbb{N},\on{LinCat}_{\on{Mod}_R})} \arrow[ld, "\on{colim}"] \\
                                                                                            & \on{LinCat}_{\on{Mod}_R} &                                                                          
\end{tikzcd}
\]
where the horizontal functor is fully faithful. Its essential image consists of those functors $B\mathbb{N}\to \on{LinCat}_{\on{Mod}_R}$ mapping $1$ to an equivalence. 
\end{enumerate}
In particular, the orbit $\infty$-category $\C/F$ of a $\on{Mod}_R$-linear equivalence $F\colon \C\to \C$ is equivalent to the colimit over a functor $B\mathbb{Z}\to \on{LinCat}_{\on{Mod}_R}$, mapping $1$ to $F$. 
\end{lemma}

\begin{proof}
We begin with showing part (1). That the inclusion $L\subset B\mathbb{N}$ is inner anodyne follows from applying \cite[4.1.2.3]{HTT} to the diagram
\[
\begin{tikzcd}
\Delta^1 \arrow[d] & {\{0,1\}} \arrow[l] \arrow[d] \arrow[r] & \ast \arrow[d] \\
\Delta^\mathbb{N}  & \mathbb{N} \arrow[l] \arrow[r]          & \ast          
\end{tikzcd}
\]
where the central $\mathbb{N}$ denotes the discrete simplicial set and $\Delta^{\mathbb{N}}$ denotes the nerve of the poset $\mathbb{N}$.
The inclusion is thus final and cofinal by \cite[Prop.~4.1.1.3]{HTT}, showing the commutativity of the diagram in (1). The horizontal functor is an equivalence by 
\cite[\href{https://kerodon.net/tag/01EJ}{Corollary 01EJ}]{Ker} and \cite[\href{https://kerodon.net/tag/01EF}{Proposition 01EF}]{Ker}

We next show part (2). The inclusion $B\mathbb{N}\subset B\mathbb{Z}$ is a weak homotopy equivalence and, using that $B\mathbb{Z}$ is a Kan complex, hence final and cofinal \cite[Cor.~4.1.2.6]{HTT}. This shows the commutativity of the diagram in (2). The fully faithfulness of the horizontal functor amounts to the statement that $B\mathbb{Z}$ is the localization of $B\mathbb{N}$ at the morphism $1$. This can be readily shown, for instance as follows. The localization of $B\mathbb{N}$ at $1$ is equivalent to the localization of $L$ at $1$, since $L\subset B\mathbb{N}$ is a categorical equivalence. This localization can be obtained as the pushout of simplicial sets $L[\{1\}^{-1}]=L\amalg_{\Delta^1}Q$, where $Q$ is a contractible Kan complex with two nondegenerate $0$-simplices, see the proof of \cite[\href{https://kerodon.net/tag/01N4}{Proposition 01N4}]{Ker}. We observe that $L[\{1\}^{-1}]$ is a Kan complex and equivalent to $B\mathbb{Z}$, since the geometric realizations of both are equivalent to the circle $S^1$.
\end{proof}

\subsection{Orbit dg-categories}\label{subsec:orbitdg}

Given a dg-category $A$, we denote by $\on{rep}_{\on{dg}}(A,A)$ the dg-category of cofibrant dg $A$-$A$-bimodules which are right quasi-representable, see also \cite{Kel06,FKQ24} for detailed definitions. We note that any such bimodule $F\in \on{rep}_{\on{dg}}(A,A)$ gives rise to a dg-functor $\on{Perf}(F)\colon \on{Perf}(A)\to \on{Perf}(A)$, and thus to a $\on{Mod}_k$-linear functor $\D(F)\colon \D(A)\to \D(A)$ between the derived $\infty$-categories. We call $F\in \on{rep}_{\on{dg}}(A,A)$ a Morita equivalence if the dg-functor $\on{Perf}(F)$ is a quasi-equivalence (or equivalently $\D(F)$ is an equivalence of $\infty$-categories).

Given $F\in \on{rep}_{\on{dg}}(A,A)$, \cite{FKQ24} defines the dg-orbit category $A/F^{\mathbb{Z}}$ as a dg-localization of the 'left lax quotient dg-category' $A/_{ll}F^{\mathbb{N}}$, which is the dg-category with the same objects as $A$ and morphism complexes $\on{Map}_{A/_{ll}F^{\mathbb{N}}}(X,Y)= \bigoplus_{i\in \mathbb{N}}\on{Map}_{A}(X,F^i(Y))$. 

\begin{remark}\label{rem:orbitHoms}
Supposing that $F\in \on{rep}_{\on{dg}}(A,A)$ is a Morita equivalence, one can show that, for all $X,Y\in A$, there is a quasi-isomorphism
\begin{equation}\label{eq:quism}\on{Map}_{A/F^{\mathbb{Z}}}(X,Y)\simeq \bigoplus_{i\in \mathbb{Z}} \on{Map}_{\on{Perf}(A)}(X,F^i(Y))\,.\end{equation}
\end{remark}

\begin{proposition}\label{prop:orbitdgcat}
Let $F\in \on{rep}_{\on{dg}}(A,\A)$ be a Morita equivalence. Then there exists a canonical equivalence of $\on{Mod}_k$-linear $\infty$-categories
\[
\D(A)/\D(F) \simeq \D(A/F^{\mathbb{Z}})
\]
between the $\on{Mod}_k$-linear orbit $\infty$-category and the derived $\infty$-category of the dg-orbit category.
\end{proposition}

\begin{remark}
We expect that $\D(A)/^{\on{lax}}\D(F)$ is equivalent to the derived $\infty$-category of the left lax quotient dg-category $A/_{ll}F^{\mathbb{N}}$. We further expect \Cref{prop:orbitdgcat} to also hold true if $F$ is not a Morita equivalence.
\end{remark}

\begin{proof}[Proof of \Cref{prop:orbitdgcat}]
As shown in \cite{FKQ24}, there exists a dg-functor $\pi\colon A\to A/F^\mathbb{Z}$ together with a natural equivalence $\pi\circ F\simeq F$. Passing to compact, cofibrant modules, $F$ induces a dg-functor $\on{Perf}(F)\colon \on{Perf}(A)\to \on{Perf}(A)$ (and not just a dg-bimodule) together with a natural equivalence $\on{Perf}(\pi)\circ \on{Perf}(F)\simeq \on{Perf}(F)$. This exhibits $\on{Perf}(A/F^{\mathbb{Z}})$ as the tip of a cocone under the functor $L\to \on{dgCat}_k[M^{-1}],\ 1\mapsto \on{Perf}(F)$. Passing to derived $\infty$-categories, this induces a $\on{Mod}_k$-linear functor $\D(A)/\D(F) \to \D(A/F^{\mathbb{Z}})$, which is checked to be an equivalence using an analogous argument as in the proof of \Cref{prop:skewgroupcatiscolim}.
\end{proof}

\begin{remark}
Suppose that $F\colon A\to A$ is a dg-functor. An a priori different definition of orbit dg-category is given in \cite{Kel05}, which we denote by $A/^{\on{dg}}F$. There is an apparent diagram of dg-categories
\[
\begin{tikzcd}
A \arrow[rr, "F"] \arrow[rd] & {} \arrow[d, Rightarrow] & A \arrow[ld] \\
                              & A/^{\on{dg}}F           &              
\end{tikzcd}
\]
where the natural transformation evaluates at $X\in A$ to the canonical morphism $F(X)\to X$ in $A/^{\on{dg}}F$. Using the universal property shown in \cite{FKQ24}, this induces a dg-functor $A/_{ll}F^{\mathbb{N}}\to A/^{\on{dg}}F$, which in turn induces a dg-functor $\alpha\colon A/F^\mathbb{Z}\to A/^{\on{dg}}F$. If $F$ is a Morita equivalence, the quasi-isomorphisms \eqref{eq:quism} imply that $\alpha$ is a a quasi-equivalence. 

In the case that $F$ is strictly invertible, the orbit dg-category $A/^{\on{dg}}F$ is furthermore isomorphic to the skew group dg-category $A\ast \mathbb{Z}$ from \Cref{sec:skewgrpdgcat}.
\end{remark}

\subsection{Example: Periodic derived categories}\label{subsec:periodicex}

Let $\on{Perf}(k)$ be the dg-category of finite dimensional chain complexes over a field $k$. Consider the invertible endofunctor $[n]\colon \on{Perf}(k)\to \on{Perf}(k)$ given by the shift functor. This defines a $\mathbb{Z}$-action on $\on{Perf}(k)$, with orbit dg-category $\on{Perf}(k)/[n]$.

Let $k[t_n]$ be the dg-algebra of graded polynomials with the monomial $t_n$ in degree $n$ (in homological grading). Note that $k[t_n]$ is the $(n+1)$-Calabi--Yau completion of $k$ in the sense of \cite{Kel11}. We similarly denote by $k[t_n^\pm]$ the dg-algebra of graded Laurent polynomials. 

\begin{proposition}\label{prop:periodicorbit}
There are equivalences in $\on{LinCat}_{\on{Mod}_k}$
\[
\D(\on{Perf}(k)/[n])\simeq \D(k)/[n] \simeq \D(k[t_n^\pm])\,.
\]
\end{proposition}

\begin{proof}
The first equivalence is the statement of \Cref{prop:orbitdgcat}. For the second equivalence, we produce a functor $\D(k)/[n] \to \D(k[t_n^\pm])$ using the universal property of the colimit. We employ the following trick, to avoid any discussions about signs for the shift functor: any $\on{Mod}_k$-linear functor $\D(k)\to \D(k[t_n^\pm])$ is fully determined by the image of $k$. Since $k[t_n^\pm][n]\simeq k[t_n^\pm]\in  \D(k[t_n^\pm])$, we thus find a commutative diagram in $\on{LinCat}_{\on{Mod}_k}$:
\[
\begin{tikzcd}
\D(k) \arrow[rd, "{k\mapsto k[t_n^\pm][n]}"'] \arrow[rr, "{[n]}"] &                  & \D(k) \arrow[ld, "{k\mapsto k[t_n^\pm]}"] \\
                                                               & {\D(k[t_n^\pm])} &                                             
\end{tikzcd}
\]
This induces a functor $\D(k)/[n] \to \D(k[t_n^\pm])$, which is now readily checked to be fully faithful and essentially surjective.
\end{proof}

\begin{remark}
Let $\D^{\on{fin}}(k[t_n])\subset \D^{\on{perf}}(k[t_n])$ denote the full subcategory of modules with finite dimensional total homology. There is an equivalence in $\on{LinCat}_{\on{Mod}_k}$
\[
\D(k[t_n])/\on{Ind}\D^{\on{fin}}(k[t_n])\simeq \D(k[t_n^\pm])\,,
\]
see \cite[Lem.~2.11]{Chr22b}. By \cite[Rem.~2.14.(2)]{HI22}, there is an equivalence of dg-categories  
\[ \on{Perf}(k)/[n]\simeq \on{Perf}(k[t_n])/\on{Perf}(k[t_n])^{\on{fin}}\,,\] 
where $\on{Perf}(k[t_n])^{\on{fin}}$ denotes the dg-category subcategory of $\on{Perf}(k[t_n])$ of dg-modules with finite dimensional total homology. Combining these equivalences provides an alternative way to prove \Cref{prop:periodicorbit}
\end{remark}

\begin{lemma}\label{lem:orbitvstensoring}
Let $\C$ be a $\on{Mod}_k$-linear $\infty$-category. Then the orbit category $\C/[n]$ is equivalent to the tensor product $\C\otimes_{\on{Mod}_k} \D(k[t_n^\pm])$ in $\on{LinCat}_{\on{Mod}_k}$.
\end{lemma}

\begin{proof}
Using that $\D(k)\simeq \on{Mod}_k$ and that the tensor product in $\on{LinCat}_{\on{Mod}_k}$ preserves colimits in the second entry, we find equivalences
\[\C\otimes_{\on{Mod}_k} \D(k[t_n^\pm])\simeq \C\otimes_{\on{Mod}_k} \on{colim}_{L}\on{Mod}_k\simeq \on{colim}_{L}\C =\C/[n]\,.\]
\end{proof}

The above allows us to obtain the following concrete description of the $1$-periodic topological Fukaya category, or equivalently Higgs category, of a marked surface studied in \cite{Chr22b}. Note that the $2$-Calabi--Yau cluster category of a marked surface is a localization of this Higgs category \cite{Wu21}, and the following result thus establishes a precise relationship between derived categories of gentle algebras and the cluster categories of surfaces. 

\begin{proposition}\label{prop:Higgs}
Let $A$ be a finite dimensional gentle algebra and ${\bf S}$ the corresponding marked surface in the sense of \cite{OPS18}. Assume that every boundary component of ${\bf S}$ contains at least one marked point\footnote{This is equivalent to the assertion that $A$ is smooth, which follows from \cite[Lem.~3.1.3]{LP20}.}. Then the Higgs category $\mathcal{H}$ \cite{Wu21} associated with ${\bf S}$, described in \cite{Chr22b}, is equivalent to the perfect derived $\infty$-category of the orbit dg-category $\on{Perf}(A)/[1]$. 

In particular, the $1$-categorical orbit category $(\on{ho}\on{Perf}(A))/[1]$ embeds fully faithfully into the triangulated homotopy $1$-category $\on{ho}\mathcal{H}$ of the Higgs category.
\end{proposition}

\begin{proof}[Proof of \Cref{prop:Higgs}.]
It is shown in \cite{Chr22b} that the $\on{Ind}$-completion of the $\infty$-categorical Higgs category $\mathcal{H}$ is equivalent to $\D(A)\otimes_{\on{Mod}_k} \D(k[t_1^\pm])$, see Proposition 4.16 and Theorem 8.4 in loc.~cit.. The statement thus reduces to \Cref{lem:orbitvstensoring}.
\end{proof}

\section{Spectral skew group algebras}\label{sec:skewgroupalg}

In this final section, we introduce a spectral version of the skew group algebra and generalize \Cref{prop:skewgroupcatiscolim} to the spectral setting.

We fix a base $\mathbb{E}_\infty$-ring spectrum $R$. Consider a $\on{Mod}_R$-linear $\infty$-category $\C$ with a $G$-action. In the case that $\C\simeq \on{RMod}_A$ has a compact generator $A$, the functor $\on{Mor}_{\C}(A,\mhyphen)\colon \C\to \on{Mod}_R$ is monadic, and hence so is the composite $\C_G\to \C\to \on{Mod}_R$. The image of $R$ under the left adjoint $\on{Mod}_R\to \C_G$ defines a compact generator $Y$ of $\C_G$. In the case that $A$ is fixed by the $G$-action, the $R$-linear endomorphism algebra of $Y$ can be considered as the corresponding skew group algebra $AG$. To be able to describe the spectral skew group algebra $AG$, we give a more systematic construction of it in \Cref{def:skewgroupspectrum}.

\begin{construction}
Let $\underline{\on{Mod}_R}\in \on{Fun}(BG,\on{LinCat}_{\on{Mod}_R})$ denote the constant functor with value $\on{Mod}_R$. We define the functor $\xi$ as the following composite:

\begin{align*}
\on{Fun}(BG,\left(\on{LinCat}_{\on{Mod}_R}\right)_{\on{Mod}_R/}) \rightarrow & \on{Fun}(BG,\on{LinCat}_{\on{Mod}_R})_{\underline{\on{Mod}_R}/}\\
 \rightarrow & \on{Fun}(BG,\on{LinCat}_{\on{Mod}_R})_{\on{Mod}_R^{\amalg G}/}\\
 \simeq &\left(\on{LinCat}_{\on{Mod}_R^{\amalg G}}\right)_{\on{Mod}_R^{\amalg G}/}\,.
\end{align*}
The $G$ action on $\on{Mod}_R^{\amalg G}$ is given by permuting the $G$-components, so that $\on{colim}_{BG} \on{Mod}_R^{\amalg G}\simeq \on{Mod}_R\in \on{LinCat}_{\on{Mod}_R}$. This equivalence is adjoint to a morphism $\on{Mod}_R^{\amalg G}\to \underline{\on{Mod}_R}$ in $\on{Fun}(BG,\on{LinCat}_{\on{Mod}_R})$ which is used for the second functor above. The third functor uses the equivalence from \Cref{prop:Gaction}.
\end{construction}

\begin{proposition}\label{prop:algBG}
There is a commutative diagram of $\infty$-categories
\[
\begin{tikzcd}[column sep=30]
{\on{Fun}(BG,\on{Alg}(\on{Mod}_R))} \arrow[r, hook, "{\on{Fun}(BG,\Theta_*)}"'] \arrow[d] & {\on{Fun}(BG,\left(\on{LinCat}_{\on{Mod}_R}\right)_{\on{Mod}_R/})} \arrow[d, "\xi"] \arrow[r] & {\on{Fun}(BG,\on{LinCat}_{\on{Mod}_R})} \arrow[d, "\simeq"] \\
{\on{Alg}(\on{Mod}_R^{\amalg G})} \arrow[r, hook, "\Theta_*"] & \left(\on{LinCat}_{\on{Mod}_R^{\amalg G}}\right)_{\on{Mod}_R^{\amalg G}/} \arrow[r]           & \on{LinCat}_{\on{Mod}_R^{\amalg G}}                        
\end{tikzcd}
\]
where the right vertical functor is the equivalence from \Cref{prop:Gaction}. We denote the left vertical functor by
\[ \xi^{\on{Alg}}\colon {\on{Fun}(BG,\on{Alg}(\on{Mod}_R))}\to {\on{Alg}(\on{Mod}_R^{\amalg G})}\,.\]
\end{proposition}

\begin{proof}
The commutativity of the right diagram follows from the definition of $\xi$. Using the fully faithfulness of the left horizontal morphisms, we can induce the left vertical functor from $\xi$ by proving that elements in the image of $\on{Fun}(BG,\Theta_*)$ are mapped by $\xi$ to elements in the image of $\Theta_*$. 

Using \cite[Prop.~4.8.5.8]{HA}, we find that a functor $BG\to \left(\on{LinCat}_{\on{Mod}_R}\right)_{\on{Mod}_R/}, \ast\mapsto (\on{Mod}_R\to \C)$ factors through $\on{Alg}(\on{Mod}_R)$ if and only if the image of $R\in \on{Mod}_R$ in $\C$, denoted $Y\in \C$, is a compact generator preserved by the $G$-action. Note that condition (6) in Proposition 4.8.5.8 is automatically fulfilled, since $\on{Mod}_R$ is generated by $R$, compare also with the proof of \cite[Thm.~7.1.2.1]{HA}.

Let $\on{Mod}_R\to \C,~R\mapsto Y$ be as before, with $Y$ a compact generator. We denote $A=\on{End}_{\C}(Y)\in \on{Alg}(\on{Mod}_R)$, and have $\C\simeq\on{RMod}_A$. The functor $\on{Mod}_R\to \C\simeq \on{RMod}_A$ is given by $(\mhyphen)\otimes A$. Its image under $\xi$ is equivalent to the functor $F\colon \on{Mod}_R^{\amalg G}\to \on{RMod}_A$, componentwise given by $(\mhyphen)\otimes A$, i.e.~satisfying $F(R_g)\simeq A$ for all $g\in G$\footnote{
We point out the curious fact that the functors $F$ and $F^R$ only depend on the $\on{Mod}_R$-linear structure of $\C$, and not on its $\on{Mod}_R^{\amalg G}$-linear structure, meaning the $G$-action on $\C$. This is because $A$ is fixed by the $G$-action on $\C$. The $G$-action will only contribute when describing the $\on{Mod}_R^{\amalg G}$-linear endomorphism algebra structure of $F^R(A)$, see \Cref{rem:skewgroupspectrum} below.}.

We again apply \cite[Prop.~4.8.5.8]{HA} to show that the functor $F:\on{Mod}_R^{\amalg G}\to \on{RMod}_A$ lies in the image of $\Theta_*$, thus showing the existence of the left commutative square. Conditions (1), (2), (3) and (5) are clear. For condition (4), it suffices to observe that the right adjoint $F^R$ of $F\colon \on{Mod}_R^{\amalg G}\to \on{RMod}_A$ is $G$-componentwise a functor that preserves geometric realizations, and thus preserves geometric realizations as well. Condition (6) boils down to the statement that $F^R$ and the counit of $F\dashv F^R$ commute with the $\on{Mod}_R^{\amalg G}$-action, which is straightforward to check. 
\end{proof}

\begin{definition}\label{def:skewgroupspectrum}
Consider an action $\rho\colon BG\to \on{Alg}(\on{Mod}_R)$ of the group $G$ on an $R$-linear ring spectrum $A$. The skew group algebra $AG\in \on{Alg}(\on{Mod}_R)$ is defined as the image of $\rho$ under the functors
\begin{equation}\label{eq:funskewgroupspectrum}
\on{Fun}(BG,\on{Alg}(\on{Mod}_R))\xlongrightarrow{\xi^{\on{Alg}}} {\on{Alg}(\on{Mod}_R^{\amalg G})}\xlongrightarrow{\on{Alg}(\psi)} \on{Alg}(\on{Mod}_R)
\end{equation}
from \Cref{rem:psi} and \Cref{prop:algBG}. 
\end{definition}

\begin{remark}\label{rem:skewgroupspectrum}
We unravel the construction of the skew group algebra $AG$. 

Starting with an $R$-linear ring spectrum $A$, with the group $G$ acting on it, we have an induced $G$-action on its $\infty$-category of (right) modules $\on{RMod}_A$, which we can also consider as a left-tensoring of $\on{RMod}_A$ by $\on{Mod}_R^{\amalg G}$. Considering $A$ as an object of $\on{RMod}_A$, we consider the $\on{Mod}_R^{\amalg G}$-linear endomorphism algebra $\tilde{A}\in \on{Alg}(\on{Mod}_R^{\amalg G})$ of $A$, in the sense of \cite[Section 4.7.1]{HA}. The algebra $\tilde{A}$ is a refinement of the skew group algebra $AG$, which is defined as the image of $\tilde{A}$ under the monoidal functor $\psi=\coprod \colon \on{Mod}_R^{\amalg G} \to \on{Mod}_R$. It remains to unravel its definition to describe the skew group algebra more explicitly. 

Given $g\in G$, the restriction to the corresponding component $\on{Mod}_R\subset \on{Mod}_R^{\amalg G}$ of $\tilde{A}$ is given by $A$. Inspecting the definition, one sees that this amounts to the fact that $A$ is preserved by the action of $g$. We write this decomposition as $\tilde{A}=\coprod_{g\in G} \tilde{A}_g$. Note that we thus have $AG\simeq \coprod_{g\in G} A$. 

We turn to the algebra structure of $\tilde{A}$ (describing it on the level of homotopy groups). To make the notation for this more transparent, we denote the $\on{Mod}_R^{\amalg G}$-action by $\otimes^G$. Firstly, we consider $\tilde{A}$ as equipped with the map $m\colon \tilde{A}\otimes^G A\to A$ in $\on{RMod}_A$, that is simply the multiplication map of $A$ on every component $A\otimes A\simeq \tilde{A}_g\otimes^G A$. The crucial point of the construction of $\tilde{A}$ in \cite[Section 4.7.1]{HA} is that the multiplication map $\tilde{m}\colon \tilde{A}\otimes^G \tilde{A}\to \tilde{A}$ in $\on{Mod}_R^{\amalg G}$ is uniquely determined by the property that the following diagram commutes\footnote{This follows from the fact that $m\colon \tilde{A}\otimes^G A\to A$ is a terminal object in the monoidal $\infty$-category $\on{Mod}_R^{\amalg G}[A]$ described in \cite[Section 4.7.1]{HA}.}:
\[
\begin{tikzcd}
\tilde{A}\otimes^G\tilde{A}\otimes^G A \arrow[d, "\tilde{m}\otimes^G A", dashed] \arrow[r, "\tilde{A}\otimes^G m"] & \tilde{A}\otimes^G A \arrow[d, "m"] \\
\tilde{A}\otimes^G A \arrow[r, "m"]                                                                          & A                                
\end{tikzcd}
\]
We write elements of the homotopy groups of $\tilde{A}$ as pairs $(g,a)$ with $g\in G$ and $a\in A$. We next argue that the morphism 
\[ \tilde{A}\otimes^G m \colon \tilde{A}_g \otimes^G \tilde{A} \otimes^G A\to \tilde{A}_g\otimes^G A\] 
is given on elements by
\[ 
(g,a)\otimes (h,b)\otimes c \mapsto (g,a)\otimes \left( (g.b)\cdot c\right)\,.
\]
The above identity follows from the following three observations:
\begin{itemize}
\item Every element $(h,b)\in \tilde{A}$ determines an inclusion $A \hookrightarrow \tilde{A}\otimes^G A$, mapping $c$ to $(h,b)\otimes c$, in $\on{RMod}_A$.
\item For every $(h,b)\in \tilde{A}$, the following diagram commutes:
\[
\begin{tikzcd}
A \arrow[d, "{c\mapsto (h,b)\otimes c}"', hook] \arrow[rd, "b\cdot (\mhyphen)"] &   \\
\tilde{A}\otimes^G A \arrow[r, "m"]                                           & A
\end{tikzcd}
\]
\item For every $b\in A$ the tensor product $R_g\otimes^G (\mhyphen)$ with the object $R_g=R$ lying in the $g$-component of $\on{Mod}_R^{\amalg G}$ maps the endomorphism $A\xrightarrow{b\cdot (\mhyphen)} A$ to the endomorphism $A\xrightarrow{g.b\cdot (\mhyphen)} A$, where $g.$ refers to the $G$-action on $A$. 
\end{itemize}
With the above, it is now straightforward to see that $\tilde{m}$ is given by 
\[
\tilde{m}\colon \tilde{A}\otimes \tilde{A} \to \tilde{A},\quad (g,a)\otimes (h,b)\mapsto (gh,a\cdot (g.b))\,,
\]
matching the desired formula of the multiplication of the skew group algebra. 

In case that $A$ is a discrete algebra object of $\on{Mod}_k$ for a commutative ring $k$, we thus see that the skew group algebra $AG$ is thus quasi-isomorphic to the classical skew group algebra.
\end{remark}

Finally, we show that the module $\infty$-category of $AG$ describes the group quotient.

\begin{proof}[Proof of \Cref{thm:colimBG}.]
Let $\tilde{A}\in \on{Alg}(\on{Mod}_R^{\amalg G})$ be the image of $A$ under the functor $\xi^{\on{Alg}}$. We have equivalences
\begin{align*}
\on{colim}_{BG}\on{RMod}_A&\simeq \psi_!(\on{RMod}_A)\\
& \simeq \psi_!( \on{RMod}_{\tilde{A}}(\on{Mod}_R^{\amalg G}))\\
& \simeq \on{RMod}_{\psi(\tilde{A})}=\on{RMod}_{AG}\,,
\end{align*}
where the first equivalence follows from \Cref{lem:colimastensorproduct}, the second from the commutativity of the left square in \Cref{prop:algBG}, and the third equivalence follows from the statement in \cite[Prop.~4.8.5.1]{HA} that the functor $\theta$ preserves coCartesian edges, see also the last paragraph in the proof of that statement.
\end{proof}

\begin{remark}
If $A$ is a dg-algebra with a strict $G$-action, we can deduce from combining \Cref{thm:colimBG} and \Cref{prop:skewgroupcatiscolim} that the dg-categorical skew group algebra $A\ast G$ is quasi-isomorphic to $AG$.
\end{remark}

\begin{remark}\label{rem:monoids}
Contrary to the results of the previous sections, the construction of the skew group algebra and the proof of \Cref{thm:colimBG} do not use that every element of $G$ has an inverse. \Cref{thm:colimBG} thus immediately generalizes to the case that $G$ is a monoid in sets. We also note that a small amount of additional work allows a generalization to the setting where $G$ is a monoid in spaces.
\end{remark}

\bibliography{biblio} 

\begin{thebibliography}{BMCSY23}

\bibitem[AP24]{AP24}
C.~Amiot and P.~Plamondon.
\newblock Skew-group ${A}_\infty$-categories as {Fukaya} categories of
  orbifolds.
\newblock {\em \href{https://arxiv.org/abs/2405.15466}{{\sf
  arXiv:2405.15466}}}, 2024.

\bibitem[Bal11]{Bal11}
P.~Balmer.
\newblock Separability and triangulated categories.
\newblock {\em Adv. Math.}, 226(5):4352--4372, 2011.

\bibitem[BMCSY23]{BCSY23}
S.~Ben-Moshe, S.~Carmeli, T.~Schlank, and L.~Yanovski.
\newblock Descent and cyclotomic redshift for chromatically localized algebraic
  {$K$}-theory.
\newblock {\em \href{https://arxiv.org/abs/2309.07123}{{\sf
  arXiv:2309.07123}}}, 2023.

\bibitem[CCRY25]{CCRY22}
S.~Carmeli, B.~Cnossen, M.~Ramzi, and L.~Yanovski.
\newblock Characters and transfer maps via categorified traces.
\newblock {\em Forum Math. Sigma}, 13:84, 2025.
\newblock Id/No e93.

\bibitem[Che15]{Che15}
X.~Chen.
\newblock A note on separable functors and monads with an application to
  equivariant derived categories.
\newblock {\em Abh. Math. Semin. Univ. Hamb.}, 85(1):43--52, 2015.

\bibitem[Chr22]{Chr22b}
M.~Christ.
\newblock Cluster theory of topological {F}ukaya categories.
\newblock {\em \href{https://arxiv.org/abs/2209.06595}{{\sf
  arXiv:2209.06595}}}, 2022.

\bibitem[Cis19]{Cis}
D.~Cisinski.
\newblock {\em Higher categories and homotopical algebra}.
\newblock Cambridge Studies in Advanced Mathematics 180, Cambridge University
  Press, 2019.

\bibitem[Coh13]{Coh13}
L.~Cohn.
\newblock Differential graded categories are k-linear stable
  $\infty$-categories.
\newblock {\em \href{https://arxiv.org/abs/1308.2587}{{\sf arXiv:1308.2587}}},
  2013.

\bibitem[COS24]{COS24}
A.~Canonaco, M.~Ornaghi, and P.~Stellari.
\newblock Localizations of the categories of $a_\infty$ categories and internal
  {H}oms over a ring.
\newblock {\em \href{https://arxiv.org/abs/2404.06610}{{\sf
  arXiv:2404.06610}}}, 2024.

\bibitem[Dem11]{Dem11}
L.~Demonet.
\newblock Categorification of skew-symmetrizable cluster algebras.
\newblock {\em Algebr. Represent. Theory}, 14(6):1087--1162, 2011.

\bibitem[Dou05]{Dou05}
C.~Douglas.
\newblock Twisted {P}arametrized {S}table {H}omotopy {T}heory.
\newblock {\em \href{https://arxiv.org/pdf/math/0508070}{{\sf arXiv:0508070}}},
  2005.

\bibitem[Ela14]{Ela14}
A.~Elagin.
\newblock On equivariant triangulated categories.
\newblock {\em \href{https://arxiv.org/abs/1403.7027}{{\sf arXiv:1403.7027}}},
  2014.

\bibitem[FKQ24]{FKQ24}
L.~Fan, B.~Keller, and Y.~Qiu.
\newblock On relative {K}oszul duality and dg enhanced orbit categories.
\newblock {\em \href{https://arxiv.org/abs/2405.00093}{{\sf
  arXiv:2405.00093}}}, 2024.

\bibitem[HI22]{HI22}
N.~Hanihara and O.~Iyama.
\newblock Enhanced {A}uslander-{R}eiten duality and {M}orita theorem for
  singularity categories, 2022.

\bibitem[HM23]{HM23}
A.~Hedenlund and T.~Moulinos.
\newblock Twisted {S}pectra {R}evisited.
\newblock {\em \href{https://arxiv.org/abs/2312.07403}{{\sf
  arXiv:2312.07403}}}, 2023.

\bibitem[Kel05]{Kel05}
B.~Keller.
\newblock On triangulated orbit categories.
\newblock {\em Doc. Math.}, 10:551--581, 2005.

\bibitem[Kel06]{Kel06}
B.~Keller.
\newblock On differential graded categories.
\newblock In {\em Proceedings of the international congress of mathematicians
  (ICM), Madrid, Spain, August 22--30, 2006. Volume II: Invited lectures},
  pages 151--190. Z{\"u}rich: European Mathematical Society (EMS), 2006.

\bibitem[Kel11]{Kel11}
B.~Keller.
\newblock Deformed {C}alabi-{Y}au completions.
\newblock {\em J. Reine Angew. Math}, 654:125--180, 2011.

\bibitem[KV00]{KV00}
M.~Kapranov and E.~Vasserot.
\newblock Kleinian singularities, derived categories and {Hall} algebras.
\newblock {\em Math. Ann.}, 316(3):565--576, 2000.

\bibitem[LP20]{LP20}
Y.~Lekili and A.~Polishchuk.
\newblock Derived equivalences of gentle algebras via {F}ukaya categories.
\newblock {\em Math. Ann.}, 376(5):187–225, 2020.

\bibitem[Lur09]{HTT}
J.~Lurie.
\newblock {\em Higher {T}opos {T}heory}.
\newblock Annals of Mathematics Studies, vol. 170, Princeton University Press,
  Princeton, NJ. MR 2522659, 2009.

\bibitem[Lur17]{HA}
J.~Lurie.
\newblock {\em Higher {A}lgebra}.
\newblock available on the author's
  {\href{https://www.math.ias.edu/~lurie/papers/HA.pdf}{{\sf webpage}}}, 2017.

\bibitem[Lur18]{SAG}
J.~Lurie.
\newblock {\em Spectral Algebraic Geometry}.
\newblock available on the author's
  {\href{https://www.math.ias.edu/~lurie/papers/SAG-rootfile.pdf}{{\sf
  webpage}}}, 2018.

\bibitem[Lur24]{Ker}
J.~Lurie.
\newblock Kerodon.
\newblock \url{https://kerodon.net}, 2024.

\bibitem[Meu20]{LM20}
P.~Le Meur.
\newblock Crossed products of {Calabi}-{Yau} algebras by finite groups.
\newblock {\em J. Pure Appl. Algebra}, 224(10):43, 2020.
\newblock Id/No 106394.

\bibitem[OPS18]{OPS18}
S.~Opper, P.~Plamondon, and S.~Schroll.
\newblock A geometric model of the derived category of a gentle algebra.
\newblock {\em \href{https://arxiv.org/abs/1801.09659}{{\sf
  arXiv:1801.09659}}}, 2018.

\bibitem[OZ22]{OZ22}
S.~Opper and A.~Zvonareva.
\newblock Derived equivalence classification of {Brauer} graph algebras.
\newblock {\em Adv. Math.}, 402:59, 2022.
\newblock Id/No 108341.

\bibitem[Pas24]{Pas24}
J.~Pascaleff.
\newblock Remarks on the equivalence between differential graded categories and
  {A}-infinity categories.
\newblock {\em Homology Homotopy Appl.}, 26(1):275--285, 2024.

\bibitem[RR85]{RR85}
I.~Reiten and C.~Riedtmann.
\newblock Skew group algebras in the representation theory of {Artin} algebras.
\newblock {\em J. Algebra}, 92:224--282, 1985.

\bibitem[VdB22]{VdB22}
M.~{V}an~{d}en {B}ergh.
\newblock Non-commutative crepant resolutions, an overview.
\newblock {\em \href{https://arXiv:2207.09703}{{\sf arXiv:2207.09703}}}, 2022.

\bibitem[Wu23]{Wu21}
Y.~Wu.
\newblock Relative cluster categories and {H}iggs categories.
\newblock {\em Adv. Math.}, 424:109040, 2023.

\end{thebibliography}
\bibliographystyle{alpha}

\textsc{Mathematisches Institut, Universität Bonn, Endenicher Allee 60, 53115 Bonn,
Germany.}

\textit{Email address:} \texttt{christ@math.uni-bonn.de}

\end{document}